\documentclass[11pt,reqno,a4paper]{amsart}
\usepackage[english]{babel}
\usepackage{amssymb,amsmath}
\usepackage[mathscr]{eucal}
\usepackage{url,ushort,bbm}
\usepackage{enumerate}
\usepackage{hyperref}
\usepackage{esvect}

\theoremstyle{plain}
\newtheorem{prop}{Proposition}[section]
\newtheorem{thm}[prop]{Theorem}
\newtheorem{cor}[prop]{Corollary}
\newtheorem{lemma}[prop]{Lemma}

\theoremstyle{definition}
\newtheorem{dfn}[prop]{Definition}
\newtheorem{remark}[prop]{Remark}
\newtheorem{example}[prop]{Example}

\newcommand{\N}{\mathbb N}
\newcommand{\Z}{\mathbb Z}
\newcommand{\R}{\mathbb R}
\newcommand{\C}{\mathbb C}
\newcommand{\HH}{\mathbb H}
\newcommand{\KK}{\mathbb K}
\newcommand{\ii}{\mathbbm i}
\newcommand{\jj}{\mathbbm j}
\newcommand{\kk}{\mathbbm k}
\newcommand{\axy}{\langle X,Y\rangle} 
\newcommand{\ax}{\langle {\ushort {X}}\rangle} 
 
\newcommand{\csim}{\stackrel{\mathrm{cyc}}{\thicksim}}
\newcommand{\ncsim}{\stackrel{\mathrm{cyc}}{\nsim}}

\DeclareMathOperator{\tr}{tr}
\DeclareMathOperator{\Tr}{Tr}
\DeclareMathOperator{\Trd}{Trd}
\DeclareMathOperator{\ran}{Ran}
\DeclareMathOperator{\sym}{Sym}

\DeclareMathOperator{\rank}{rank}

\DeclareMathOperator{\inte}{int}

\DeclareMathOperator{\spann}{span}
\newcommand{\mato}{\left(\begin{smallmatrix}}
\newcommand{\matc}{\end{smallmatrix}\right)}

\renewcommand{\emptyset}{\varnothing}
\renewcommand{\setminus}{\smallsetminus}

\newcommand{\eps}{\varepsilon}

\begin{document}
\title{The truncated tracial moment problem}
\author[Sabine Burgdorf]{Sabine Burgdorf}
\address{Sabine Burgdorf, Universit\"at Konstanz, 
Fachbereich Mathematik und Sta\-tistik, 78457 Konstanz, Germany \\ \and
Institut de Recherche Math\'ematique de Rennes, Universit\'e de Rennes 1,
Campus de Beaulieu, 35042 Rennes cedex, France} 
\email{sabine.burgdorf@uni-konstanz.de}
\thanks{Both authors were supported by the French-Slovene partnership project Proteus 20208ZM.
The first author was partially supported by the Zukunftskolleg Konstanz.
 The second author was partially supported by the Slovenian Research Agency (program no. 1-0222).
}
\author[Igor Klep]{Igor Klep}
\address{Igor Klep,
Univerza v Ljubljani, Fakulteta za matematiko in fiziko,
Jadranska 19, 1111 Ljubljana \\ and
Univerza v Mariboru, Fakulteta za naravoslovje in
matematiko, Koro\v ska 160, 2000 Maribor, Slovenia}
\email{igor.klep@fmf.uni-lj.si}

\date{16 January 2010}

\numberwithin{equation}{section}

\begin{abstract}
We present \textit{tracial} analogs of the
classical results of Curto and Fialkow on moment matrices.
A sequence of real numbers indexed by words in non-commuting variables
with values invariant under cyclic permutations of the indexes,
is called a \textit{tracial sequence}. We prove
that such a sequence can be represented with tracial moments of matrices
if its corresponding moment matrix is positive semidefinite and of finite
rank. A \emph{truncated} tracial sequence allows for such a representation
if and only if one of its extensions admits a flat extension. Finally, we apply
this theory via duality to investigate trace-positive polynomials in non-commuting variables.
\end{abstract}

\maketitle
\vspace{-2pt}
\section{Introduction}

The moment problem is a classical question in analysis, well studied because of its 
importance and variety of applications. A simple example is the (univariate) Hamburger 
moment problem: when does a given sequence of real numbers represent the successive
moments $\int\! x^n\, d\mu(x)$ of a positive Borel measure $\mu$ on $\R$?
Equivalently, which linear functionals $L$ on univariate real polynomials are
integration with respect to some $\mu$? By Haviland's theorem \cite{Hav}
this is the case if and only if $L$ is nonnegative on all polynomials nonnegative on 
$\R$. Thus Haviland's theorem relates the moment problem to positive polynomials. It
holds in several variables and also if we are interested in restricting the support of 
$\mu$. For details we refer the reader to one of the many beautiful expositions of this 
classical branch of functional analysis, e.g.~\cite{Akh,KN,ST}.

Since Schm\"udgen's celebrated solution of the moment problem
on compact basic closed semialgebraic sets \cite{Smu},
the moment problem has played a prominent role in real algebra,
exploiting this duality between positive polynomials and the
moment problem, cf.~\cite{KM,PS,Put,PV}.
The survey of Laurent \cite{laurent2} gives a nice presentation of
up-to-date results and applications;
see also \cite{Mar,PD} for more on positive polynomials.

Our main motivation are trace-positive polynomials in non-commuting
variables. A polynomial is called \emph{trace-positive} if all
its matrix evaluations (of \emph{all} sizes) have nonnegative trace.
Trace-positive polynomials have been employed to investigate
problems on 
operator algebras (Connes' embedding conjecture \cite{connes,ksconnes})
and mathematical physics (the Bessis-Moussa-Villani conjecture 
\cite{bmv,ksbmv}), so a good understanding of this set is desired.
By duality this leads us to consider the tracial moment problem
introduced below.
We mention that the free non-commutative moment problem 
has been studied and solved by
McCullough \cite{McC} and Helton \cite{helton}.
Hadwin \cite{had} considered
moments involving traces on von Neumann algebras.

This paper is organized as follows. The short Section \ref{sec:basic}
fixes notation and terminology involving non-commuting variables used in the sequel. 
ection \ref{sec:ttmp} introduces 
tracial moment sequences,
tracial moment matrices,
the tracial moment problem, and their truncated counterparts.
Our main results in this section relate the truncated tracial moment problem
to flat extensions of tracial moment matrices and resemble the 
results of Curto and Fialkow \cite{cffinite,cfflat} on the (classical)
truncated moment problem. For example,
we prove
that a tracial sequence can be represented with tracial moments of 
matrices
if its corresponding tracial moment matrix is positive semidefinite and of finite
rank (Theorem \ref{thm:finiterank}). 
A truncated tracial sequence allows for such a representation
if and only if one if its extensions admits a flat extension (Corollary
\ref{cor:flatt}).
Finally, in Section \ref{sec:poly} we 
explore the duality
between the tracial moment problem and trace-positivity of polynomials.
Throughout the paper several examples are given
to illustrate the theory.

\section{Basic notions}\label{sec:basic}

Let $\R\ax$ denote the unital associative $\R$-algebra freely generated 
by $\ushort X=(X_1,\dots,X_n)$. The elements of $\R\ax$ are polynomials in the non-commuting 
variables $X_1,\dots,X_n$ with coefficients in $\R$. 
An element $w$ of the monoid $\ax$, freely generated by $\ushort X$, 
is called a \textit{word}. An element of the form $aw$, where $0\neq a\in\R$ 
and  $w\in\ax$, is called a \textit{monomial} and $a$ its \textit{coefficient}.
We endow $\R\ax$ with the \textit{involution} $p\mapsto p^*$ fixing $\R\cup\{\ushort X\}$ 
pointwise. Hence for each word $w\in\ax$, $w^*$ is its reverse. As an example, we have 
$(X_1X_2^2-X_2X_1)^*=X_2^2X_1-X_1X_2$. 

For $f\in\R\ax$ we will substitute symmetric matrices
$\ushort A=(A_1,\dots A_n)$ of the same size for the variables $\ushort X$ 
and obtain a matrix $f(\ushort A)$. Since $f(\ushort A)$ is
not well-defined if the $A_i$ do not have the 
same size, we will assume this condition implicitly without further mention in the sequel. 

Let $\sym \R\ax$ denote the set of \emph{symmetric elements} in $\R\ax$, i.e., 
$$\sym \R\ax=\{f\in \R\ax\mid f^*=f\}.$$
Similarly, we use $\sym \R^{t\times t}$ to denote the set of all symmetric $t\times t$ matrices. 

In this paper we will mostly consider the \emph{normalized} trace $\Tr$,
i.e.,
$$\Tr(A)=\frac 1t\tr(A)\quad\text{for } A\in\R^{t\times t}.$$
The invariance of the trace under cyclic permutations motivates the
following definition of cyclic equivalence \cite[p.~1817]{ksconnes}. 

\begin{dfn}
Two polynomials $f,g\in \R\ax$ are \emph{cyclically equivalent}
if $f-g$ is a sum of commutators:
$$f-g=\sum_{i=1}^k(p_iq_i-q_ip_i) \text{ for some } k\in\N
\text{ and } p_i,q_i \in \R\ax.$$
\end{dfn}

\begin{remark}\label{rem:csim}
\mbox{}\par
\begin{enumerate}[(a)]
\item Two words $v,w\in\ax$ are cyclically equivalent if and only if $w$ 
is a cyclic permutation of $v$. 
Equivalently: there exist $u_1,u_2\in\ax$ such that 
$v=u_1u_2$ and $w=u_2u_1$.
\item If $f\csim g$ then $\Tr(f(\ushort A))=\Tr(g(\ushort A))$ for all tuples
$\ushort A$ of symmetric matrices.
Less obvious is the converse: if $\Tr(f(\ushort A))=\Tr(g(\ushort A))$
for all $\ushort A$ and $f-g\in\sym\R\ax$, then $f\csim g$ \cite[Theorem 2.1]{ksconnes}.
\item Although $f\ncsim f^*$ in general, we still have 
$$\Tr(f(\ushort A))=\Tr(f^*(\ushort A))$$
for all $f\in\R \ax$ and all $\ushort A\in (\sym\R^{t\times t})^n$.
\end{enumerate}
\end{remark}

The length of the longest word in a polynomial $f\in\R\ax$ is the
\textit{degree} of $f$ and is denoted by $\deg f$.
We write $\R\ax_{\leq k}$ for the set of all polynomials of degree $\leq k$.

\section{The truncated tracial moment problem}\label{sec:ttmp}

In this section we define tracial (moment) sequences,
tracial moment matrices,
the tracial moment problem, and their truncated analogs.
After a few motivating examples we proceed to show that the 
kernel of a tracial moment matrix has some real-radical-like
properties (Proposition \ref{prop:radical}). 
We then prove that a tracial moment matrix of finite
rank has a tracial moment representation, i.e., the tracial moment problem
for the associated tracial sequence is solvable (Theorem \ref{thm:finiterank}). 
Finally, we give the solution of 
the truncated tracial moment problem: a truncated tracial sequence has
a tracial representation if and only if one of its extensions has a tracial moment matrix that 
admits a flat extension (Corollary \ref{cor:flatt}).

For an overview of the classical (commutative) moment problem in several 
variables we refer 
the reader to Akhiezer \cite{Akh}  (for the analytic theory) and
to the survey of Laurent \cite{laurent} and references therein for a more
algebraic approach.
The standard references on the truncated moment problems are
\cite{cffinite,cfflat}.
For the non-commutative moment problem with \emph{free} (i.e.,
unconstrained) moments see
\cite{McC,helton}.

\begin{dfn}
A sequence of real numbers $(y_w)$ indexed by words $w\in \ax$ satisfying 
\begin{equation}
	y_w=y_u \text{ whenever } w\csim u, \label{cyc}
\end{equation}
\begin{equation}
	y_w=y_{w^*} \text{ for all } w, \label{cycstar}
\end{equation}
and $y_\emptyset=1$, is called a (normalized) \emph{tracial sequence}. 
\end{dfn} 

\begin{example}
Given $t\in\N$ and symmetric matrices $A_1,\dots,A_n\in \sym \R^{t\times t}$,
the sequence given by $$y_w:= \Tr(w(A_1,\dots,A_n))=\frac 1t \tr(w(A_1,\dots,A_n))$$ 
is a tracial sequence since by Remark \ref{rem:csim}, the traces of cyclically 
equivalent words coincide. 
\end{example}

We are interested in the converse of this example (the \emph{tracial moment problem}): 
\emph{For which sequences $(y_w)$ do there exist $N\in \N$, $t\in \N$,
$\lambda_i\in \R_{\geq0}$ with $\sum_i^N \lambda_i=1$ and 
vectors $\ushort A^{(i)}=(A_1^{(i)},\dots,A_n^{(i)})\in (\sym \R^{t\times t})^n$,  such that
\begin{equation}
	y_w=\sum_{i=1}^N \lambda_i \Tr(w(\ushort A^{(i)}))\,? \label{rep}
\end{equation}}
We then say that $(y_w)$ has a \emph{tracial moment representation}
and call it a \emph{tracial moment sequence}.

The \emph{truncated tracial moment problem} is the study of (finite) tracial sequences 
$(y_w)_{\leq k}$ 
where $w$ is constrained by $\deg w\leq k$ for some $k\in\N$,
and properties \eqref{cyc} and \eqref{cycstar} hold for these $w$.
For instance, which sequences $(y_w)_{\leq k}$ have a tracial moment 
representation, i.e., when does there 
exist a representation of the values $y_w$ as in \eqref{rep}  for $\deg w\leq k$? 
If this is the case, then 
the sequence $(y_w)_{\leq k}$ is called a \emph{truncated tracial moment sequence}.

\begin{remark}
\mbox{}\par
\begin{enumerate}[(a)]
\item 
To keep a perfect analogy with the classical moment problem, 
one would need to consider the existence of a positive
Borel measure $\mu$ on $(\sym \R^{t\times t})^n$ (for some
$t\in\N$) satisfying
\begin{equation}\label{eq:gewidmetmarkus}
y_w = \int \! w(\ushort A) \, d\mu(\ushort A).
\end{equation}
As we shall mostly focus on the \emph{truncated}
tracial moment problem in the sequel, the
finitary representations \eqref{rep} seem to be the
proper setting. 
We look forward to studying the more general representations
\eqref{eq:gewidmetmarkus} in the future.
\item
Another natural extension of our tracial moment problem
with respect to matrices would be to consider moments obtained by 
traces in finite \emph{von Neumann algebras} as
done by Hadwin \cite{had}.
However, our
primary motivation were trace-positive polynomials
defined via traces of matrices (see Definition \ref{def:trpos}),
a theme we expand upon in Section \ref{sec:poly}. Understanding these
is one of the approaches to Connes' embedding conjecture \cite{ksconnes}.
The notion dual to that of trace-positive polynomials is
the tracial moment problem as defined above.
\item The tracial moment problem
is a natural extension of the classical quadrature problem
dealing with 
representability via atomic positive measures in
the commutative case. Taking $\ushort a^{(i)}$
consisting of $1\times 1$ matrices $a_j^{(i)}\in\R$ 
for the $\ushort A^{(i)}$ 
in \eqref{rep},  we have
$$y_w=\sum_i \lambda_i w(\ushort a^{(i)})= \int \!x^w \, d\mu(x),$$
where $x^w$ denotes the commutative collapse of $w\in\ax$.
The measure $\mu$ is the convex combination 
$\sum \lambda_i\delta_{\ushort a^{(i)}}$
of the atomic measures $\delta_{\ushort a^{(i)}}$.
\end{enumerate}
\end{remark}

The next example shows that there are (truncated) tracial moment sequences $(y_w)$ 
which
cannot be written as $$y_w=\Tr(w(\ushort A)).$$ 

\begin{example}\label{exconv} 
Let $X$ be a single free (non-commutative) variable.
We take the index set $J=(1,X,X^2,X^3,X^4)$ and $y=(1,1-\sqrt2,1,1-\sqrt2,1)$. Then
$$y_w=\frac{\sqrt2}{2}w(-1)+(1-\frac{\sqrt2}{2})w(1),$$ i.e., 
$\lambda_1=\frac{\sqrt2}{2}$, $\lambda_2=1-\lambda_1$ and $A^{(1)}=-1$, $A^{(2)}=1$. 
But there is no symmetric matrix $A\in \R^{t\times t}$ for any $t\in\N$ such that 
$y_w=\Tr(w(A))$ for all $w\in J$. The proof is given in the appendix.
\end{example}

The (infinite) \emph{tracial moment matrix} $M(y)$ of a tracial 
sequence $y=(y_w)$ is defined by 
$$M(y)=(y_{u^*v})_{u,v}.$$
This matrix is symmetric due to the condition \eqref{cycstar} in the 
definition of a tracial sequence.
A necessary condition for $y$ to be a tracial moment sequence is positive 
semidefiniteness of $M(y)$ which in general is not sufficient.

The tracial moment matrix of \emph{order $k$} is the tracial moment matrix $M_k(y)$ 
indexed by words $u,v$ with $\deg u,\deg v\leq k$.
If $y$ is a truncated tracial moment sequence, then $M_k(y)$ is positive 
semidefinite. Here is an easy example showing the converse is false:

\begin{example}\label{expsd}
When dealing with two variables, we write $(X,Y)$ instead of $(X_1,X_2)$.
Taking the index set 
$$(1,X,Y,X^2,XY,Y^2,X^3,X^2Y,XY^2,Y^3,X^4,X^3Y,X^2Y^2,XYXY,XY^3,Y^4)$$
the truncated moment sequence $$y=(1,0,0,1,1,1,0,0,0,0,4,0,2,1,0,4) $$ yields the 
tracial moment matrix 
$$M_2(y)=\left(\begin{smallmatrix}
	1&0&0&1&1&1&1\\ 0&1&1&0&0&0&0\\ 0&1&1&0&0&0&0\\ 1&0&0&4&0&0&2\\
	1&0&0&0&2&1&0\\ 1&0&0&0&1&2&0\\ 1&0&0&2&0&0&4
\end{smallmatrix}\right)$$
with respect to the basis $(1,X,Y,X^2,XY,YX,Y^2)$. 
$M_2(y)$ is positive semidefinite but $y$ has no tracial representation.
Again, we postpone the proof until the appendix.
\end{example}

For a given polynomial $p=\sum_{w\in \ax} p_w w\in \R \ax$ let $\vv p$ be the
(column) vector of coefficients $p_w$ in a given fixed order.
One can identify $\R \ax_{\leq k}$ with $\R^\eta$
for $\eta=\eta(k)=\dim\R\ax_{\leq k}<\infty$ by sending each $p\in \R \ax_{\leq k}$ to the vector 
$\vv p$ of its entries with $\deg w\leq k$. 
The tracial moment matrix $M(y)$ induces the linear map 
$$\varphi_M:\R\ax\to \R^\N,\quad p\mapsto M\vv p.$$ The tracial moment matrices $M_k(y)$, 
indexed by $w$ with $\deg w\leq k$, can be regarded as linear maps 
$\varphi_{M_k}:\R^\eta\to \R^\eta$, $\vv p\mapsto M_k\vv p$.

\begin{lemma}\label{lem:mk}
Let $M=M(y)$ be a tracial moment matrix. Then the following holds:
\begin{enumerate}[\rm (1)]
\item $p(y):=\sum_w p_w y_w={\vv{1}}^*M\vv{p}$. In particular,
	${\vv{1}}^*M\vv{p}={\vv{1}}^*M\vv{q}$ if $p\csim q$;
\item ${\vv{p}}^*M\vv{q}={\vv{1}}^*M\vv{p^*q}$.
 \end{enumerate}
\end{lemma}

\begin{proof}
Let $p,q\in \R \ax$. For $k:=\max \{\deg p,\deg q\}$, we have 
\begin{equation}
{\vv{p}}^*M(y)\vv{q}={\vv{p}}^*M_k(y)\vv{q}.
\end{equation}
Both statements now follow by direct calculation. 
\end{proof}

We can identify the kernel of a tracial moment matrix $M$ with the subset of $\R \ax$
given by 
\begin{equation}\label{eq:momKer}
	I:=\{p\in \R \ax\mid M\vv p=0\}.
\end{equation}

\begin{prop}\label{lem:kerideal} Let $M\succeq0$ be a tracial moment matrix. Then 
	\begin{equation}\label{kerideal}
	I=\{p\in \R \ax\mid \langle M\vv{p},\vv{p}\rangle=0\}.
	\end{equation} 
	Further, $I$ 
	is a two-sided ideal of $\R \ax$ invariant under the involution. 
\end{prop}
\begin{proof}
	Let $J:=\{p\in \R \ax\mid \langle M\vv{p},\vv{p}\rangle=0\}$. The implication
$I\subseteq J$ is obvious. Let $p\in J$ be given and $k=\deg p$.
Since $M$ and thus $M_k$ for each $k\in \N$ is positive semidefinite, the square root 
$\sqrt{M_k}$ of $M_k$ exists. Then
$0=\langle M_k\vv{p},\vv p\rangle=\langle\sqrt{M_k}\vv{p}, \sqrt{M_k}\vv{p}\rangle$ implies
$\sqrt{M_k}\vv{p}=0$. This leads to $M_k\vv{p}=M\vv p=0$, thus $p\in I$.

To prove that $I$ is a two-sided ideal, it suffices to show that $I$ is a right-ideal 
which is closed under *. To do this, consider the bilinear map 
$$ \langle p,q\rangle_M:= \langle M\vv{p},\vv{q}\rangle$$ on $\R \ax$, which is a semi-scalar 
product. By Lemma \ref{lem:mk}, we get that
$$\langle pq,pq\rangle_M=((pq)^*pq)(y)=(qq^*p^*p)(y)= \langle pqq^*,p\rangle_M.$$
Then by the Cauchy-Schwarz inequality it follows that for $p\in I$, we have
$$0\leq \langle pq,pq\rangle_M^2=\langle pqq^*,p\rangle_M^2\leq 
\langle pqq^*,pqq^*\rangle_M\langle p,p\rangle_M=0.$$
Hence $pq\in I$, i.e., $I$ is a right-ideal.

Since $p^*p\csim pp^*$, we obtain from Lemma \ref{lem:mk} that
$$\langle M\vv{p},\vv{p} \rangle=\langle p,p \rangle_M=(p^*p)(y)=(pp^*)(y)=\langle p^*,p^* 
\rangle_M=
\langle M{\vv p}^*,{\vv p}^* \rangle.$$ Thus if $p\in I$ then also $p^*\in I$. 
\end{proof}

In the \emph{commutative} context, the kernel of $M$ is a real radical ideal if $M$ is positive
semidefinite as observed by Scheiderer (cf.~\cite[p.~2974]{laurent2}).
The next proposition gives a description of 
the kernel of $M$ in the non-commutative setting, and could be helpful in 
defining a non-commutative real radical ideal.

\begin{prop}\label{prop:radical}
For the ideal $I$ in \eqref{eq:momKer} we have
$$I=\{f\in \R \ax\mid (f^*f)^k\in I \;\text{for some}\;k\in \N\}.$$ 
Further, 
$$I=\{f\in \R \ax\mid (f^*f)^{2k}+\sum g_i^*g_i\in I \;\text{for some}\;k\in \N, g_i\in \R \ax\}.
$$
\end{prop}

\begin{proof}
If $f\in I$ then also $f^*f\in I$ since $I$ is an ideal. If $f^*f\in I$ we have
$M\vv{f^*f}=0$ which implies by Lemma \ref{lem:mk} that
$$0={\vv 1}^*M\vv{f^*f}={\vv f}^*M\vv{f}=\langle Mf,f\rangle.$$ 
Thus $f\in I$. 
If $(f^*f)^k\in I$ then also $(f^*f)^{k+1}\in I$. So without loss of generality let $k$ be even. 
From $(f^*f)^k\in I$ we obtain 
$$0={\vv 1}^*M\vv{(f^*f)^k}={\vv{(f^*f)^{k/2}}}^*M\vv{(f^*f)^{k/2}},$$ implying 
$(f^*f)^{k/2}\in I$. This leads to $f\in I$ by induction.

To show the second statement let $(f^*f)^{2k}+\sum g_i^*g_i\in I$. This leads to
$${\vv{(f^*f)^k}}^*M\vv{(f^*f)^k}+\sum_i {\vv{g_i}}^*M\vv{g_i}=0.$$ Since 
$M(y)\succeq0$ we have ${\vv{(f^*f)^k}}^*M\vv{(f^*f)^k}\geq 0$ and 
${\vv{g_i}}^*M\vv{g_i}\geq 0.$ Thus ${\vv{(f^*f)^k}}^*M\vv{(f^*f)^k}=0$ 
(and ${\vv{g_i}}^*M\vv{g_i}= 0$) which implies $f\in I$ as above. 
\end{proof}

In the commutative setting one uses  the Riesz representation theorem for 
some set of continuous functions (vanishing at infinity or with compact support) 
to show the existence of a representing measure. We will use the Riesz 
representation theorem for positive linear functionals on a 
finite-dimensional Hilbert space. 

\begin{dfn}
Let $\mathcal A$ be an $\R$-algebra with involution. We call a linear map 
$L:\mathcal A\to \R$ a \emph{state} if 
$L(1)=1$, $L(a^*a)\geq0$ and $L(a^*)=L(a)$ for all $a\in\mathcal A$. 
If all the commutators have value $0$, i.e., if $L(ab)=L(ba)$ for all 
$a,b\in \mathcal A$, then $L$ is called a \emph{tracial state}.
\end{dfn}

With the aid of the Artin-Wedderburn theorem we shall
characterize tracial states on matrix $*$-algebras in Proposition
\ref{prop:convtrace}.
This will enable us to prove the existence of a tracial moment representation for
tracial sequences with a finite rank tracial moment matrix; see Theorem
\ref{thm:finiterank}.

\begin{remark}\label{rem:aw}
The only central simple algebras over $\R$ are full matrix
algebras over $\R$, $\C$ or $\HH$ (combine the Frobenius theorem 
\cite[(13.12)]{Lam} with the Artin-Wedderburn theorem \cite[(3.5)]{Lam}).
In order to understand ($\R$-linear) tracial states on these, we recall
some basic Galois theory.

Let 
$$\Trd_{\C/\R}:\C\to\R, \quad z\mapsto\frac 12(z+\bar z) $$ 
denote the \emph{field trace} and 
$$\Trd_{\HH/\R}:\HH\to\R,\quad z\mapsto\frac12(z+\bar z)$$
the \emph{reduced trace} \cite[p.~5]{boi}.
Here the Hamilton quaternions $\HH$ are endowed with the \emph{standard
involution}
$$
z=a+\ii b+\jj c+\kk d \mapsto a-\ii b-\jj k-\kk d = \bar z
$$
for $a,b,c,d\in\R$.
We extend the canonical involution on $\C$ and $\HH$ to the conjugate
transpose involution $*$ on matrices
over $\C$ and $\HH$, respectively.

Composing the field trace and reduced trace, respectively, with the normalized
trace, yields an $\R$-linear map from $\C^{t\times t}$ and
$\HH^{t\times t}$, respectively, to $\R$. We will denote it simply
by $\Tr$. A word of \emph{caution}: 
$\Tr(A)$ does not denote the (normalized) matricial trace 
over $\KK$
if $A\in \KK^{t\times t}$ and $\KK\in\{\C,\HH\}$.
\end{remark}

An alternative description of $\Tr$ is given by the following lemma:

\begin{lemma}\label{lem:convtrace}
Let $\KK\in\{\R,\C,\HH\}$. Then
the only $(\R$-linear$)$ tracial state on $\KK^{t\times t}$ is $\Tr$.
\end{lemma}

\begin{proof}
An easy calculation shows that $\Tr$ is indeed a tracial state.

Let $L$ be a tracial state on $\R^{t\times t}$.
By the Riesz representation theorem there exists a positive 
semidefinite matrix $B$ with $\Tr(B)=1$ such that $$L(A)=\Tr(BA)$$ for all 
$A\in\R^{t\times t}$.

Write $B=\begin{pmatrix}b_{ij}\end{pmatrix}_{i,j=1}^{t}$.
Let 
$i\neq j$.
Then $A=\lambda E_{ij}$ has zero trace for every 
$\lambda\in \R$ and is thus a sum of commutators. 
(Here $E_{ij}$ denotes the $t\times t$ \emph{matrix unit} with a one
in the $(i,j)$-position and zeros elsewhere.)
Hence 
$$\lambda b_{ij} = L(A)  = 0.$$
Since $\lambda\in\R$ was arbitrary, $b_{ij}=0$.

Now let $A=\lambda (E_{ii}-E_{jj})$. Clearly, 
$\Tr(A)=0$ and hence $$\lambda(b_{ii}-b_{jj})= L(A)= 0.$$
As before, this gives $b_{ii}=b_{jj}$. So $B$ is scalar,
and $\Tr(B)=1$. Hence it is the
identity matrix. In particular, $L=\Tr$.

If $L$ is a tracial state on $\C^{t\times t}$, 
then $L$ induces a tracial state on $\R^{t\times t}$,
so $L_0:=L|_{\R^{t\times t}}=\Tr$ by the above.
Extend $L_0$ to 
$$L_1:\C^{t\times t} \to \R,
\quad A+\ii B\mapsto L_0(A)=\Tr(A)  \quad\text{for } A,B\in\R^{t\times t}.
$$
$L_1$ is a tracial state on $\C^{t\times t}$ as a 
straightforward computation
shows. As $\Tr(A)=\Tr(A+\ii B)$, all we need to show is that $L_1=L$.

Clearly, $L_1$ and $L$ agree on the vector space spanned
by all commutators in $\C^{t\times t}$. This space is (over $\R$)
of codimension $2$. By construction, $L_1(1)=L(1)=1$ and
$L_1(\ii)=0$. On the other hand,
$$L(\ii)=L(\ii^*)=-L(\ii)$$ implying $L(\ii)=0$.
This shows $L=L_1=\Tr$.

The remaining case of tracial states over $\HH$ is dealt 
with
similarly and is left as an exercise for the reader.
\end{proof}

\begin{remark}\label{rem:real}
Every complex number $z=a+\ii b$ can be represented
as a $2\times 2$ real matrix 
$z'=\left(\begin{smallmatrix} a & b \\ -b & a\end{smallmatrix}\right)$.
This gives rise to 
an $\R$-linear $*$-map
$\C^{t\times t}\to \R^{(2t)\times(2t)}$ that commutes with $\Tr$. 
A similar property holds if quaternions
$a+\ii b+\jj c+\kk d$ 
are represented by the $4\times 4$ real matrix
$$\left(\begin{smallmatrix}
 a & b & c & d \\ 
 -b & a & -d & c \\
 -c & d & a & -b \\
 -d & -c & b & a 
\end{smallmatrix}\right).$$
\end{remark}

\begin{prop}\label{prop:convtrace}
	Let $\mathcal A$ be a $*$-subalgebra of $ \R^{t\times t}$ for some $t\in \N$ and
	$L:\mathcal A\to \R$ a tracial state.
	Then there exist 
full matrix algebras $\mathcal A^{(i)}$ over $\R$, $\C$ or $\HH$, 
a $*$-isomorphism 
\begin{equation}\label{eq:iso}
\mathcal A\to\bigoplus_{i=1}^N \mathcal A^{(i)},
\end{equation}
and $\lambda_1,\dots, \lambda_N\in \R_{\geq0}$ with $\sum_i \lambda_i=1$, such that for all 
$A\in \mathcal A$,
	 $$L(A)=\sum_i^N \lambda_i\Tr(A^{(i)}).$$
Here, $\bigoplus_i A^{(i)} =\left(\begin{smallmatrix} A^{(1)} \\ & \ddots \\ & & A^{(N)}
\end{smallmatrix}\right)$ denotes the image of $A$ under the isomorphism
\eqref{eq:iso}. The size of $($the real representation of$)$ $\bigoplus_i A^{(i)}$ is 
at most $t$.
\end{prop}

\begin{proof}
Since $L$ is tracial, 
$L(U^*AU)=L(A)$ for all orthogonal $U\in\R^{t\times t}$.
Hence we can apply orthogonal transformations to $\mathcal A$ 
without changing the values of $L$. 
So $\mathcal A$ can be transformed into block diagonal form
as in \eqref{eq:iso}
according to its invariant subspaces.
That is, each of the blocks $\mathcal A^{(i)}$ 
acts irreducibly on a subspace of $\R^t$ and is thus 
a central 
simple algebra (with involution) over $\R$.
The involution on $\mathcal A^{(i)}$ is induced by the
conjugate transpose involution. (Equivalently, by the
transpose on the real matrix representation in the complex
of quaternion case.)

Now $L$ induces (after a possible normalization) a tracial state on the block
$\mathcal A^{(i)}$ and hence by Lemma \ref{lem:convtrace}, we have
$L_i:=L|_{\mathcal A^{(i)}}=\lambda_i \Tr$ for some $\lambda_i\in\R_{\geq0}$.
Then
\[
L(A)=L\big(\bigoplus_i A^{(i)}\big)=\sum_i L_i\big(A^{(i)}\big)
= \sum_i \lambda_i \Tr\big(A^{(i)}\big)
\]
and
$1=L(1)=\sum_i \lambda_i$.
\end{proof}

The following theorem is the tracial version of the representation theorem 
of Curto and Fialkow for moment matrices with finite rank \cite{cffinite}.

\begin{thm}\label{thm:finiterank}
Let $y=(y_w)$ be a tracial sequence with positive semidefinite 
moment matrix $M(y)$ of finite rank $t$. Then $y$ is a tracial moment
sequence, i.e., there exist vectors 
$\ushort A^{(i)}=(A_1^{(i)},\dots,A_n^{(i)})$ of symmetric matrices $A_j^{(i)}$ 
of size at most $t$ and $\lambda_i\in \R_{\geq0}$ with $\sum \lambda_i=1$ 
such that $$y_w=\sum \lambda_i \Tr(w(\ushort A^{(i)})).$$ 
\end{thm}

\begin{proof}
Let $M:=M(y)$. We equip $\R\ax$ with the bilinear form given by
$$\langle p,q\rangle_M:=\langle M\vv{p},\vv{q} \rangle={\vv{q}}^*M\vv p.$$ Let 
$I=\{p\in \R\ax\mid \langle p,p\rangle_M=0\}.$ Then by Proposition \ref{lem:kerideal}, 
$I$ is an ideal of $\R \ax$. In particular, $I=\ker \varphi_M$ for 
$$\varphi_M:\R \ax\to \ran M,\quad p\mapsto M\vv{p}.$$ Thus if we define
$E:=\R \ax/I$, the induced linear map 
$$\overline\varphi_M:E\to \ran M,\quad \overline p\mapsto M\vv{p}$$
is an isomorphism and $$\dim E=\dim(\ran M)=\rank M=t<\infty.$$ Hence 
$(E,\langle$\textvisiblespace ,\textvisiblespace $\rangle_E)$ is a finite-dimensional 
Hilbert space for
$\langle \bar p,\bar q\rangle_E={\vv{q}}^*M\vv{p}$. 

Let $\hat X_i$ be the right multiplication with $X_i$ on $E$, i.e., 
$\hat X_i \overline p:=\overline{pX_i}$. Since 
$I$ is a right ideal of $\R \ax$, the operator $\hat X_i$ is well defined.
Further, $\hat X_i$ is symmetric since
\begin{align*}
\langle \hat X_i \overline p,\overline q \rangle_E&=\langle M \vv{pX_i},\vv{q} \rangle
= (X_ip^*q)(y)\\
&=(p^*qX_i)(y)=\langle M \vv{p},\vv{qX_i} \rangle=\langle\overline p,\hat X_i\overline q \rangle_E.
\end{align*}
Thus each $\hat X_i$, acting on a $t$-dimensional vector space, has a representation matrix 
$A_i\in \sym \R^{t\times t}$.
 
Let $\mathcal B=B(\hat X_1,\dots,\hat X_n)=B(A_1,\dots,A_n)$ be the algebra of 
operators generated by $\hat X_1,\dots,\hat X_n$. These operators can be written
as $$\hat p=\sum_{w\in\ax} p_w \hat{w}$$ for some $p_w\in \R$, 
where $\hat w=\hat X_{w_1}\cdots \hat X_{w_s}$ for $w=X_{w_1}\cdots X_{w_s}$.
Observe that $\hat{w}=w(A_1,\dots,A_n)$.
We define the linear functional $$L:\mathcal B\to\R,\quad
\hat p\mapsto {\vv{1}}^*M\vv p=p(y),$$
which is a state on $\mathcal B$.
Since $y_w=y_u$ for $w\csim u$, it follows that $L$ is tracial. Thus by Proposition 
\ref{prop:convtrace} (and Remark \ref{rem:real}), there exist
$\lambda_1,\dots \lambda_N\in \R_{\geq0}$ with $\sum_i\lambda_i=1$ and real symmetric matrices $A_j^{(i)}$
$(i=1,\ldots,N$) 
for each $A_j\in \sym \R^{t\times t}$, such that for all $w\in \ax$, 
$$y_w=w(y)=L(\hat w)=\sum_i \lambda_i \Tr(w(\ushort A^{(i)})),$$
as desired. 
\end{proof}

The sufficient conditions on $M(y)$ in Theorem \ref{thm:finiterank} are also 
necessary for $y$ to be a tracial moment sequence. Thus we get our first 
characterization of tracial moment sequences:

\begin{cor}\label{cor:finite}
Let $y=(y_ w)$ be a tracial sequence. Then $y$ is a tracial moment sequence
if and only if $M(y)$ is positive semidefinite and of finite rank.
\end{cor}

\begin{proof}
If $y_ w=\Tr( w(\ushort A))$ for some $\ushort A=(A_1,\dots,A_n)\in(\sym \R^{t\times t})^n$, 
then $$L(p)=\sum_ w p_ w y_ w=\sum_ w p_ w \Tr( w(\ushort A))=
 \Tr(p(\ushort A)).$$
Hence 
\begin{align*}
{\vv p}^*M(y)\vv{p}&=L(p^*p)=\Tr(p^*(\ushort A)p(\ushort A))\geq0.
\end{align*} 
for all $p \in \R\ax$. 

Further, the tracial moment matrix $M(y)$ has rank at most $t^2$.
This can be seen as follows: 
$M$ induces a bilinear map 
$$\Phi:\R \ax\rightarrow\R \ax^*,\quad p\mapsto\Big(q\mapsto \Tr\big((q^*p)(\ushort A)\big)\Big),$$
where $\R \ax^*$ is the dual space of $\R \ax$. This implies 
$$\rank M=\dim (\ran\Phi)=\dim(\R \ax/\ker\Phi).$$
The kernel of the evaluation map 
$\eps_{\ushort A}:\R\ax\rightarrow\R^{t\times t}$, $p\mapsto p(\ushort A)$
is a subset of $\ker \Phi$. In particular, 
\[\dim(\R\ax/\ker\Phi)\leq \dim(\R\ax/\ker\eps_{\ushort A})=\dim(\ran \eps_{\ushort A})\leq t^2. \]
The same holds true for each convex combination $y_w=\sum_i \lambda_i \Tr( w(\ushort A^{(i)}))$.

The converse is Theorem \ref{thm:finiterank}.
\end{proof}

\begin{dfn}\label{defflat}
Let $A\in \sym\R^{t\times t}$ be given. A (symmetric) extension of $A$ is a matrix 
$\tilde A\in \sym\R^{(t+s)\times (t+s)}$ of the form
$$\tilde A=\begin{pmatrix} A &B \\ B^* & C\end{pmatrix} $$
for some $B\in \R^{t\times s}$ and $C\in \R^{s\times s}$.
Such an extension is \emph{flat} if $\rank A=\rank\tilde A$,
or, equivalently, if $B = AW$  and $C = W^*AW$ for some matrix $W$.
\end{dfn}

The kernel of a flat extension $M_k$ of a tracial moment matrix $M_{k-1}$ 
has some (truncated) \emph{ideal-like properties} as 
shown in the following lemma.  

\begin{lemma}\label{lem:flatrideal}
Let $f\in \R \ax$ with $\deg f\leq k-1$ and let $M_k$ be a flat extension of $M_{k-1}$. 
If $f\in\ker M_k$ then $fX_i,X_if\in \ker M_k$.
\end{lemma}

\begin{proof}
Let $f=\sum_w f_w w$. Then for $v\in \ax_{k-1}$, we have
\begin{equation}\label{eqker}
(M_k\vv{fX_i})_v =\sum_w f_w y_{v^*wX_i}= 
\sum_w f_w y_{(vX_i)^*w}=(M_k \vv f)_{vX_i}=0.
\end{equation}

The matrix $M_k$ is of the form $M_k=\mato M_{k-1}&B\\B^*&C\matc$. 
Since $M_k$ is a flat extension, 
$\ker M_k=\ker \begin{pmatrix} M_{k-1}&B\end{pmatrix}$. 
Thus by \eqref{eqker},  
$fX_i\in \ker \begin{pmatrix} M_{k-1}&B\end{pmatrix}=\ker M_k$. 
For $X_if$ we obtain analogously that
$$(M_k\vv{X_if})_v =\sum_w f_w y_{v^*X_iw}=
\sum_w f_w y_{(X_iv)^*w}=(M_k \vv f)_{X_iv}=0$$
for $v\in \ax_{k-1}$, which implies $X_if\in \ker M_k$.
\end{proof}

We are now ready to prove the tracial version of the flat extension theorem of
Curto and Fialkow \cite{cfflat}.

\begin{thm}\label{thm:flatextension}
Let $y=(y_w)_{\leq 2k}$ be a truncated tracial sequence of order $2k$. If 
$\rank M_k(y)=\rank M_{k-1}(y)$, then there exists
a unique tracial extension $\tilde y=(\tilde y_w)_{\leq 2k+2}$ of $y$ such that 
$M_{k+1}(\tilde y)$ is a flat extension of $M_k(y)$.
\end{thm}

\begin{proof}
Let $M_k:=M_k(y)$.
We will construct a flat extension $M_{k+1}:=\mato M_k&B\\B^*&C\matc$ 
such that $M_{k+1}$ is a tracial moment matrix. Since 
$M_k$ is a flat extension of $M_{k-1}(y)$ we can find a basis $b$ of 
$\ran M_k$ consisting of columns of $M_k$ labeled by $w$ with $\deg w\leq k-1$.
Thus the range of $M_k$ is completely determined by the range of $M_k|_{\spann b}$, 
i.e., for each $p\in \R \ax$ with $\deg p\leq k$ there exists a \emph{unique} 
$r\in \spann b$ such that
$M_k\vv p=M_k \vv r$; equivalently, $p-r\in \ker M_k$. 

Let $v\in\ax$, $\deg v=k+1$, $v=v'X_i$ for some $i\in \{1,\dots,n\}$ and $v'\in \ax$ 
with $\deg v'=k$. 
For $v'$ there exists an $r\in \spann b$ such that $v'-r\in \ker M_k$. 

\emph{If} there exists a flat extension $M_{k+1}$, then by Lemma \ref{lem:flatrideal},
from $v'-r\in \ker M_k\subseteq\ker M_{k+1}$ it
follows that $(v'-r)X_i\in \ker M_{k+1}$. Hence the desired flat extension 
has to satisfy 
\begin{equation}\label{eqflatcond}
	M_{k+1}\vv{v}=M_{k+1}\vv{rX_i}=M_k\vv{rX_i}.
\end{equation}
Therefore we define 
\begin{equation}\label{eq:sabinedefinesB}
B\vv{v}:=M_k\vv{rX_i}.
\end{equation}

More precisely, let $(w_1,\dots,w_\ell)$ be the 
basis of $M_k$, i.e., $(M_k)_{i,j}=w_i^*w_j$. Let $r_{w_i}$
be the unique element in $\spann  b$ with $ w_i-r_{ w_i}\in \ker M_k$.
Then $B=M_kW$ with 
$W=(r_{ w_1X_{i_1}},\dots,r_{ w_\ell X_{i_\ell}})$ and we define 
\begin{equation}\label{eq:sabinedefinesC}
C:=W^*M_kW. 
\end{equation}
Since the $r_{ w_i}$ are uniquely determined, 
\begin{equation}\label{eq:sabinedefinesMk+1}
M_{k+1}=\mato M_k&B\\B^*&C\matc
\end{equation} 
is well-defined. The constructed $M_{k+1}$ is a flat extension of 
$M_k$, and 
$M_{k+1}\succeq0$ if and only if $M_k\succeq0$, cf.~\cite[Proposition 2.1]{cfflat}.
Moreover, once $B$ is chosen, there is only one $C$ making
$M_{k+1}$ as in \eqref{eq:sabinedefinesMk+1} a flat extension of $M_k$. 
This follows from general
linear algebra, see e.g.~\cite[p.~11]{cfflat}. Hence $M_{k+1}$ is the 
\emph{only} candidate for a flat extension.

Therefore we are done if $M_{k+1}$ is a tracial moment matrix, i.e., 
\begin{equation}
 (M_{k+1})_w=(M_{k+1})_v \;\text{ whenever}\; w\csim v. \label{mm}
\end{equation}
To show this we prove that $(M_{k+1})_{X_iw}=(M_{k+1})_{wX_i}$. Then \eqref{mm} 
follows recursively. 

Let $w=u^*v$. If $\deg u,\deg vX_i\leq k$ there is nothing to show since 
$M_k$ is a tracial moment matrix. If $\deg u\leq k$ and $\deg vX_i=k+1$ there exists
an $r\in \spann  b$ such that $r-v\in \ker M_{k-1}$, and by Lemma \ref{lem:flatrideal},
also $vX_i-rX_i\in \ker M_k$. Then we get
\begin{align*}
(M_{k+1})_{u^*vX_i}&=\vv{u}^*M_{k+1}\vv{vX_i}=\vv{u}^*M_{k+1}\vv{rX_i}
=\vv{u}^*M_{k}\vv{rX_i}\\
&=(M_k)_{u^*rX_i}
=(M_k)_{X_iu^*r}
=(M_k)_{(uX_i)^*r}\\
&\overset{(\ast)}{=}{\vv{uX_i}}^*M_{k+1}\vv{v}=(M_{k+1})_{(uX_i)^*v}
=(M_{k+1})_{X_iw},
\end{align*}
where equality $(\ast)$ holds by \eqref{eqflatcond} which implies Lemma 
\ref{lem:flatrideal} by construction.

If $\deg u=\deg vX_i=k+1$, write $u=X_ju'$. Further, there exist $s,r\in \spann b$ with 
$u'-s\in \ker M_{k-1}$ and $r-v\in \ker M_{k-1}$. Then 
\begin{align*}
(M_{k+1})_{u^*vX_i}&=\vv{X_ju'}^*M_{k+1}\vv{vX_i}=\vv{X_js}^*M_{k}\vv{rX_i}\\
&=(M_k)_{s^*X_jrX_i}=(M_k)_{(sX_i)^*(X_jr)}\\
&\overset{(*)}{=}\vv{uX_i}^*M_{k+1}\vv{X_jv}=(M_{k+1})_{(uX_i)^*X_jv}
=(M_{k+1})_{X_i w}.
\end{align*}

Finally, the construction of $\tilde y$ from $M_{k+1}$ is clear. 
\end{proof}

\begin{cor}\label{cor:flat}
Let $y=(y_ w)_{\leq 2k}$ be a truncated tracial sequence. If 
$M_k(y)$ is positive semidefinite
and $M_k(y)$ is a flat extension of $M_{k-1}(y)$, then $y$ 
is a truncated tracial moment sequence.
\end{cor}

\begin{proof}
	By Theorem \ref{thm:flatextension} we can extend $M_k(y)$ inductively
	to a positive semidefinite moment matrix $M(\tilde y)$ with 
	$\rank M(\tilde y)=\rank M_k(y)<\infty$. Thus $M(\tilde y)$ has finite 
	rank and by Theorem \ref{thm:finiterank}, there exists a tracial moment 
	representation 
	of $\tilde y$. Therefore $y$ is a truncated tracial moment sequence.
\end{proof}

The following two corollaries give characterizations of tracial 
moment matrices coming from tracial moment sequences.

\begin{cor}\label{cor:flatall}
Let $y=(y_ w)$ be a tracial sequence. Then $y$ 
is a tracial moment sequence if and only if $M(y)$ is positive semidefinite and there
exists some $N\in \N$ such that $M_{k+1}(y)$ is a flat extension of 
$M_{k}(y)$ for all $k\geq N$.
\end{cor}

\begin{proof}
If $y$ is a tracial moment sequence then by Corollary \ref{cor:finite},
$M(y)$ is positive semidefinite and has finite rank $t$. Thus there exists an 
$N\in \N$ such that $t=\rank M_N(y)$. 
In particular, $\rank M_k(y)=\rank M_{k+1}(y)=t$ for all $k\geq N$, i.e., $M_{k+1}(y)$ 
is a flat extension of $M_k(y)$ for all $k\geq N$.

For the converse, let $N$ be given such that $M_{k+1}(y)$ is a flat extension of 
$M_{k}(y)$ for all $k\geq N$. By Theorem \ref{thm:flatextension}, the (iterated) 
unique extension $\tilde y$ of $(y_w)_{\leq 2k}$ for $k\geq N$ is equal to $y$.
Otherwise there exists a flat extension $\tilde y$ of $(y_w)_{\leq 2\ell}$ 
for some $\ell\geq N$ such that $M_{\ell+1}(\tilde y)\succeq 0$ is a flat extension
of $M_\ell(y)$ and $M_{\ell+1}(\tilde y)\neq M_{\ell+1}(y)$ contradicting the 
uniqueness of the extension in Theorem \ref{thm:flatextension}.

Thus $M(y)\succeq 0$ and $\rank M(y)=\rank M_N(y)<\infty$. Hence by Theorem \ref{thm:finiterank}, 
$y$ is a tracial moment sequence.
\end{proof}

\begin{cor}\label{cor:flatt}
Let $y=(y_ w)$ be a tracial sequence. Then $y$ 
has a tracial moment representation with matrices of size at most 
$t:=\rank M(y)$ if 
$M_N(y)$ is positive semidefinite and $M_{N+1}(y)$ is 
a flat extension of $M_{N}(y)$ for some $N\in \N$ with $\rank M_N(y)=t$.
\end{cor}

\begin{proof}
Since $\rank M(y)=\rank M_N(y)=t,$
each $M_{k+1}(y)$ with $k\geq N$ is a flat extension of $M_k(y)$.
As $M_N(y)\succeq0$, all $M_k(y)$ 
are positive semidefinite. 
Thus $M(y)$ is also positive semidefinite. Indeed, let
$p\in\R\ax$ 
and $\ell=\max\{\deg p,N\}$. Then 
${\vv p}^*M(y)\vv p={\vv p}^*M_\ell(y)\vv p\geq0$.

Thus by Corollary \ref{cor:flatall}, $y$ is a tracial moment sequence. The 
representing matrices can be chosen to be of size at most $\rank M(y)=t$. 
\end{proof}

\section{Positive definite moment matrices and trace-positive polynomials}\label{sec:poly}

In this section we explain how the representability of \emph{positive definite}
tracial moment matrices relates 
to sum of hermitian squares representations of
trace-positive polynomials. We start by introducing some terminology.

An element of the form $g^*g$ for some $g\in\R\ax$ is called a 
\textit{hermitian square} and we denote the set of all sums of hermitian 
squares by 
$$\Sigma^2=\{f\in\R\ax\mid f=\sum g_i^*g_i \;\text{for some}\; g_i\in\R\ax\}.$$
A polynomial $f\in \R \ax$ is \emph{matrix-positive} if $f(\ushort A)$ is positive 
semidefinite for all tuples $\ushort A$ of symmetric matrices 
$A_i\in \sym \R^{t\times t}$, $t\in\N$. Helton \cite{helton} proved that $f\in\R\ax$ is 
matrix-positive if and only if $f\in \Sigma^2$ by solving a non-commutative 
moment problem; see also \cite{McC}.

We are interested in a different type of positivity induced by
the trace.
\begin{dfn}\label{def:trpos}
A polynomial $f\in \R \ax$ is called \emph{trace-positive} if 
$$\Tr(f(\ushort A))\geq 0\;\text{ for all}\; \ushort A\in(\sym\R^{t\times t})^n,\; t\in\N.$$
\end{dfn}

Trace-positive polynomials are intimately connected to deep open
problems from 
e.g.~operator algebras (Connes' embedding conjecture \cite{ksconnes})
and mathematical physics (the Bessis-Moussa-Villani conjecture 
\cite{ksbmv}), so a good understanding of this set is needed.
A distinguished subset is formed by sums of hermitian squares and
commutators.

\begin{dfn}
Let $\Theta^2$ be the set of all polynomials which are cyclically 
equivalent to a sum of hermitian squares, i.e.,
\begin{equation}\label{eq:defcycsohs}
\Theta^2=\{f\in \R\ax\mid f\csim\sum g_i^*g_i\;\text{for some}\;g_i \in\R\ax\}.
\end{equation}
\end{dfn}

Obviously, all $f\in \Theta^2$ are trace-positive. However, in contrast to 
Helton's sum of squares theorem mentioned above, the following
non-commutative version of the well-known Motzkin polynomial \cite[p.~5]{Mar} shows that 
a trace-positive polynomial need not be a member of $\Theta^2$ \cite{ksconnes}. 

\begin{example}\label{motznc}
Let $$M_{\rm nc}=XY^4X+YX^4Y-3XY^2X+1\in\R\axy.$$ Then $M_{\rm nc}\notin \Theta^2$ since 
the commutative Motzkin polynomial is not a (commutative) sum of squares \cite[p.~5]{Mar}. 
The fact that $M_{\rm nc}(A,B)$ has nonnegative trace for all symmetric matrices $A,B$ 
has been shown by Schweighofer and the second author \cite[Example 4.4]{ksconnes} using  
Putinar's
Positivstellensatz \cite{Put}.
\end{example}

Let $\Sigma_k^2:=\Sigma^2\cap \R \ax_{\leq 2k}$ and $\Theta_k^2:=\Theta^2\cap \R \ax_{\leq 2k}$.
These are convex cones in $\R \ax_{\leq 2k}$. 
By duality there exists a connection 
between $\Theta_k^2$ and positive semidefinite tracial moment matrices of order $k$. 
If every tracial moment matrix $M_k(y)\succeq0$ of order $k$ has a tracial representation 
then every trace-positive polynomial of degree at most $2k$ lies in $\Theta_k^2$. 
In fact:

\begin{thm}\label{thm:posdefmm}
	The following statements are equivalent:
	\begin{enumerate}[\rm (i)]
	\item all truncated tracial sequences $(y_ w)_{\leq 2k}$ with 
	{\rm{positive definite}} tracial moment matrix $M_k(y)$ have a tracial moment representation \eqref{rep};
	\item all trace-positive polynomials of degree $\leq2k$ are elements of $\Theta^2_k$.
	\end{enumerate}
\end{thm}

For the proof we need some preliminary work.
\begin{lemma}\label{lem:thetaclosed}
	$\Theta_k^2$ is a closed convex cone in $\R \ax_{\leq 2k}$.
\end{lemma}

\begin{proof}
Endow $\R\ax_{\leq 2k}$ with a norm 
$\|$\textvisiblespace $\|$ and  the quotient space $\R \ax_{\leq 2k}/_{\csim}$ 
with the quotient norm
\begin{equation}\label{eq:qnorm}
\| \pi(f) \| := \inf \big\{ \| f+h \| \mid h\csim 0\big\}, \quad
f\in\R\ax_{\leq 2k}.
\end{equation}
Here $\pi:\R\ax_{\leq 2k}\to \R \ax_{\leq 2k}/_{\csim}$ denotes 
the quotient map. (Note: due to the finite-dimensionality of $\R\ax_{\leq 2k}$,
the infimum on the right-hand side of \eqref{eq:qnorm} is attained.)		

Since $\Theta_k^2= \pi^{-1} \big( \pi(\Theta_k^2)\big)$, it suffices
to show that $\pi(\Theta_k^2)$ is closed.
Let $d_k=\dim \R \ax_{\leq 2k}$. Since by Carath\`eodory's theorem \cite[p.~10]{bar} each element
	$f\in \R \ax_{\leq 2k}$ can be written as a convex combination of $d_k+1$ elements
	of $\R \ax_{\leq 2k}$, the image of
\begin{align*}
\varphi:\left(\R \ax_{\leq k}\right)^{d_k}
&\to
\R \ax_{2k}/_{\csim}\\
(g_i)_{i=0,\dots,d_k}
&\mapsto 
\pi\big(\sum_{i=0}^{d_k}g_i^*g_i\big)
\end{align*}
equals $\pi(\Sigma^2_k)=\pi(\Theta_k^2)$. In $\left(\R \ax_{\leq k}\right)^{d_k}$ we define 
$\mathcal S:=\{g=(g_i)\mid \|g\|=1\}$. Note that $\mathcal S$ is compact, thus 
$V:=\varphi(\mathcal S)\subseteq \pi(\Theta_k^2)$ is compact as well. 
Since $0\notin \mathcal S$,
and a sum of hermitian squares cannot be cyclically equivalent to $0$ by 
\cite[Lemma 3.2 (b)]{ksbmv}, we see that
$0\notin V$.

Let $(f_\ell)_\ell$ be a sequence in $\pi(\Theta^2_k)$ which converges to $\pi(f)$ 
for some $f\in\R \ax_{\leq 2k}$. 
Write $f_\ell=\lambda_\ell v_\ell$ for $\lambda_\ell\in\R_{\geq 0}$ and $v_\ell\in V$. 
Since $V$ is compact there exists a subsequence $(v_{\ell_j})_j$ of $v_\ell$ converging 
to $v\in V$. Then
$$\lambda_{\ell_j}=\frac{\|f_{\ell_j}\|}{\|v_{\ell_j}\|}\stackrel{j\rightarrow \infty}{\longrightarrow }\frac{\|f\|}{\|v\|}.$$ 
Thus $f_\ell\rightarrow f=\frac{\|f\|}{\|v\|}v\in\pi(\Theta^2_k)$.
\end{proof}

\begin{dfn}
To a truncated tracial sequence $(y_ w)_{\leq k}$ we
associate
the \emph{$($tracial$)$ Riesz functional} $L_y:\R \ax_{\leq k}\to\R$ defined by
$$L_y(p):=\sum_ w p_ w y_ w\quad\text{for } p=\sum_ w p_ w  w\in \R\ax_{\leq k}.$$
We say that $L_y$ is \emph{strictly positive} ($L_y>0$), if 
$$L_y(p)>0 \text{ for all trace-positive } p\in\R \ax_{\leq k},\, p\ncsim 0.$$ 
If $L_y(p)\geq0$ for all trace-positive $p\in\R \ax_{\leq k}$, then
$L_y$ is \emph{positive} ($L_y\geq0$).
\end{dfn}

Equivalently, a tracial Riesz functional $L_y$
is positive (resp., strictly positive) if and only if the map 
$\bar L_y$ it induces on $ \R \ax_{\leq 2k}/_{\csim}$ is 
nonnegative (resp., positive) on 
the nonzero images of trace-positive polynomials in $ \R \ax_{\leq 2k}/_{\csim}$.

We shall prove that strictly positive Riesz functionals lie in the interior of the cone
of positive Riesz functionals,
and that truncated tracial sequences $y$ with \emph{strictly}
positive $L_y$ are truncated tracial moment sequences (Theorem \ref{thm:Lrep} below). 
These results are motivated by and resemble the 
results of Fialkow and Nie 
\cite[Section 2]{fnie} in the commutative context.

\begin{lemma}\label{lem:Linner}
	If $L_y>0$ then there exists an $\eps>0$ such that $L_{\tilde y}>0$ for all
	$\tilde y$ with $\|y-\tilde y\|_1<\eps$.
\end{lemma}

\begin{proof}
We equip $\R \ax_{\leq 2k}/_{\csim}$ with a quotient norm as in  \eqref{eq:qnorm}.
Then $$\mathcal S:=\{\pi(p)\in \R \ax_{\leq 2k}/_{\csim}\mid p\in\mathcal C_k,\;\|\pi(p)\|=1\}$$ is compact.
By a scaling argument, it suffices to show that $\bar L_{\tilde y}>0$ on $\mathcal S$ for $\tilde y$ close to $y$. 
The map $y\mapsto \bar L_y$ is linear between finite-dimensional vector spaces.
Thus 
$$|\bar L_{y'}(\pi(p))-\bar L_{y''}(\pi(p))|\leq C \|y'-y''\|_1$$ for all $\pi(p)\in \mathcal S$, 
truncated tracial moment sequences $y',y''$, and some $C\in\R_{>0}$.

Since $\bar L_y$ is continuous and strictly positive on $\mathcal S$,
 there exists an $\varepsilon>0$ such 
that $\bar L_y(\pi(p))\geq2\varepsilon$ for all $\pi(p)\in \mathcal S$. 
Let $\tilde y$ satisfy $\|y-\tilde y\|_1<\frac {\eps}C$.
Then
\[\bar L_{\tilde y}(\pi(p))\geq \bar L_y(\pi(p))-C \|y-\tilde y\|_1\geq\varepsilon>0. \hfill\qedhere \]
\end{proof}

\begin{thm}\label{thm:Lrep}
	Let $y=(y_ w)_{\leq k}$ be a truncated tracial sequence of order $k$.
	If $L_y>0$, then $y$ is  a truncated tracial moment sequence.
\end{thm}

\begin{proof}
We show first that 
$y\in \overline T$, where $\overline T$ is the closure of
$$T=\big\{(y_ w)_{\leq k}\mid  \exists \ushort A^{(i)}\;\exists \lambda_i\in \R_{\geq0} :\; y_ w=\sum \lambda_i\Tr( w(\ushort A^{(i)}))\big\}.$$

Assume $L_y>0$ but $y\notin \overline T$. Since $\overline T$ is a closed 
convex cone in $\R^\eta$ (for some $\eta\in \N$), by the Minkowski separation 
theorem there exists a vector $\vv{p}\in \R^\eta$ such that $\vv{p}^*y<0$ 
and $\vv{p}^*w\geq 0$ for all $w\in \overline T$. The non-commutative 
polynomial corresponding to $\vv{p}$ is
trace positive since $\vv{p}^*z\geq 0$ for all $z\in \overline T$. Thus
$0<L_y(p)=\vv{p}^*y<0$, a contradiction.	

By Lemma \ref{lem:Linner}, $y\in\inte(\overline T)$. Thus $y\in \inte (\overline T)\subseteq T$
\cite[Theorem 25.20]{ber}.
\end{proof}

We remark that assuming only non-strict positivity of $L_y$ in Theorem \ref{thm:Lrep}
would not suffice for the existence of a tracial moment representation \eqref{rep}
for $y$. This is a consequence of Example \ref{expsd}.

\begin{proof}[Proof $(\!$of Theorem {\rm\ref{thm:posdefmm}}$)$]
To show (i) $\Rightarrow$ (ii), assume $f=\sum_ w f_ w  w
\in\R\ax_{\leq 2k}$ is 
trace-positive but $f\notin \Theta^2_k$. 
By Lemma \ref{lem:thetaclosed}, $\Theta_k^2$ is a closed convex cone in $\R\ax_{\leq 2k}$, thus
by the Minkowski separation theorem we find a hyperplane which 
separates $f$ and $\Theta_k^2$. That is, there is a linear form 
$L:\R\ax_{\leq 2k}\to\R$ such that $L(f)<0$ and $L(p)\geq0$ 
for $p\in \Theta_k^2$. In particular, $L(f)=0$ for all $f\csim 0$, i.e.,
without loss of generality, $L$ is tracial. 
Since there are tracial states strictly positive on $\Sigma^2_k\setminus\{0\}$, we may assume $L(p)>0$ 
for all $p\in \Theta_k^2$, $p\ncsim 0$.
Hence
the bilinear form given by $$(p,q)\mapsto L(pq)$$ can be written as  
$ L(pq)={\vv q}^*M\vv{p}$ for some truncated tracial moment matrix $M\succ0$. 
By assumption, the corresponding truncated tracial sequence
$y$ has a tracial moment representation $$y_ w=\sum \lambda_i \Tr( w(\ushort A^{(i)}))$$
for some tuples $A^{(i)}$ of symmetric matrices $A_j^{(i)}$ and $\lambda_i\in \R_{\geq0}$ 
which implies the contradiction
$$0>L(f)=\sum \lambda_i \Tr(f(\ushort A^{(i)}))\geq 0.$$ 

Conversely, if (ii) holds,
then $L_y>0$ if and only if $M(y)\succ0$. Thus a positive definite moment matrix $M(y)$
defines a strictly positive functional $L_y$ which by Theorem \ref{thm:Lrep} has a tracial
representation. 
\end{proof}

As mentioned above, the Motzkin polynomial $M_{\rm nc}$
is trace-positive but $M_{\rm nc}\notin \Theta^2$. Thus by Theorem \ref{thm:posdefmm}
there exists at least one truncated tracial moment matrix which is positive definite but has  
no tracial representation. 
\begin{example}
Taking the index set 
$$(1,X,Y,X^2,XY,YX,Y^2,X^2Y,XY^2,YX^2,Y^2X,X^3,Y^3,XYX,YXY),$$
 the 
matrix 
$$M_3(y):=\left(\begin{smallmatrix}
1 & 0  & 0 & \frac74 & 0 & 0 & \frac74 & 0 & 0 & 0 & 0 & 0 & 0 & 0 & 0 \\ 
0 & \frac74 & 0 & 0 & 0 & 0 & 0 & 0 & \frac{19}{16} & 0 & \frac{19}{16} & \frac{21}4 & 0 & 0 & 0 \\ 
0 & 0 & \frac74 & 0 & 0 & 0 & 0 & \frac{19}{16} & 0 & \frac{19}{16} & 0 & 0 & \frac{21}4 & 0 & 0 \\ 
\frac74 & 0 & 0 & \frac{21}4 & 0 & 0 & \frac{19}{16} & 0 & 0 & 0 & 0 & 0 & 0 & 0 & 0 \\ 
0 & 0 & 0 & 0 & \frac{19}{16} & 0 & 0 & 0 & 0 & 0 & 0 & 0 & 0 & 0 & 0 \\ 
0 & 0 & 0 & 0 & 0 & \frac{19}{16} & 0 & 0 & 0 & 0 & 0 & 0 & 0 & 0 & 0 \\ 
\frac74 & 0 & 0 & \frac{19}{16} & 0 & 0 &\frac{21}4 & 0 & 0 & 0 & 0 & 0 & 0 & 0 & 0 \\ 
0 & 0 & \frac{19}{16} & 0 & 0 & 0 & 0 & \frac{9}8 & 0 & \frac{5}6 & 0 & 0 & \frac{9}8 & 0 & 0 \\ 
0 & \frac{19}{16} & 0 & 0 & 0 & 0 & 0 & 0 & \frac{9}8 & 0 & \frac{5}6 & \frac{9}8 & 0 & 0 & 0 \\ 
0 & 0 & \frac{19}{16} & 0 & 0 & 0 & 0 & \frac{5}6 & 0 & \frac{9}8 & 0 & 0 & \frac{9}8 & 0 & 0 \\ 
0 & \frac{19}{16} & 0 & 0 & 0 & 0 & 0 & 0 & \frac{5}6 & 0 & \frac{9}8 & \frac{9}8 & 0 & 0 & 0 \\ 
0 & \frac{21}4 & 0 & 0 & 0 & 0 & 0 & 0 & \frac{9}8 & 0 & \frac{9}8 & 51 & 0 & 0 & 0 \\ 
0 & 0 & \frac{21}4 & 0 & 0 & 0 & 0 & \frac{9}8 & 0 & \frac{9}8 & 0 & 0 & 51 & 0 & 0 \\ 
0 & 0 & 0 & 0 & 0 & 0 & 0 & 0 & 0 & 0 & 0 & 0 & 0 & \frac{5}6 & 0 \\ 
0 & 0 & 0 & 0 & 0 & 0 & 0 & 0 & 0 & 0 & 0 & 0 & 0 & 0 & \frac{5}6
\end{smallmatrix} \right)$$
is a tracial moment matrix of degree 3 in 2 variables and is positive definite.
But $$L_y(M_{\rm nc})=M_{\rm nc}(y)=-\frac5{16}<0.$$ Thus $y$ 
is not a truncated tracial moment sequence,
since otherwise $L_y(p)\geq 0$ for all trace-positive polynomials $p\in \R\axy_{\leq 6}$.

On the other hand, the (free) non-commutative moment problem is always
solvable for positive definite moment matrices \cite[Theorem 2.1]{McC}.
In our example this means
there are symmetric matrices $A,B\in\R^{15\times 15}$ and a vector
$v\in\R^{15}$ such that
$$y_ w=\langle w(A,B)v,v\rangle$$
for all $ w\in\axy_{\leq 3}$.
\end{example}

\begin{remark}
A trace-positive polynomial $f\in \R \ax$ of degree $2k$ lies in $\Theta^2_k$ if
and only if $L_y(f)\geq 0$ for all truncated tracial sequences $(y_w)_{\leq 2k}$ with
 $M_k(y)\succeq0$. 
This condition is obviously satisfied if all truncated tracial sequences $(y_w)_{\leq 2k}$ with 
$M_k(y)\succeq0$ have a tracial representation. 

Using this we can prove that trace-positive binary quartics, i.e., 
homogeneous polynomials of degree $4$ in $\R \langle X,Y\rangle$, lie in $\Theta_2^2$.
Equivalently, truncated tracial sequences $(y_w)$ indexed by words of degree $4$ with a 
positive definite tracial 
moment matrix have a tracial moment representation.

Furthermore,
trace-positive binary biquadratic polynomials, i.e., polynomials $f\in \R \axy$ with 
$\deg_X f, deg_Y f\leq 2$,
are cyclically equivalent to a sum of hermitian squares. 
Example \ref{expsd} then shows that a polynomial $f$ can satisfy $L_y(f)\geq 0$ although there 
are  truncated tracial sequences $(y_w)_{\leq 2k}$ with $M_k(y)\succeq0$ and no 
tracial representation.

Studying extremal points of the convex cone $$\{(y_w)_{\leq 2k}\mid M_k(y)\succeq 0\}$$ 
of truncated tracial sequences with positive semidefinite tracial moment matrices, we are able 
to impose a concrete block structure on the matrices needed in a tracial moment representation. 

These statements and concrete sum of hermitian squares and commutators representations of trace-positive polynomials 
of low degree will be published elsewhere \cite{sb}. 
\end{remark}

\appendix
\section{Proofs of the claims made in Examples \ref{exconv} and \ref{expsd}}
\subsection*{Example \ref{exconv} revisited}

We take the index set $J=(1,X,X^2,X^3,X^4)$ and $y=(1,1-\sqrt2,1,1-\sqrt2,1)$. Then
there is no symmetric matrix $A\in \R^{t\times t}$ for any $t\in \N$ such that 
\begin{equation}\label{eq:meanRep}
y_ w=\Tr( w(A))\quad \text{for all } w\in J.
\end{equation}

Without loss of generality we can choose $A$ to be diagonal with diagonal elements 
$a_1,\dots, a_t$. 
Then $y_ w=\Tr( w(A))$ if and only if the following equations hold:
\begin{align}
\sum_{i=1}^t a_i&=\sum_{i=1}^t a_i^3=(1-\sqrt2)t,  \label{eq:mean1} \\
 \sum_{i=1}^t a_i^2&= \sum_{i=1}^t a_i^4=t \label{eq:mean2}.
\end{align}
In the general means inequality 
$$\frac{\sum_{i=1}^t x_i}{t}\geq\sqrt{\frac{\sum_{i=1}^t x_i^2}{t}}$$ for the arithmetic 
and the quadratic mean of $x=(x_1,\dots, x_t)\in\R_{\geq0}^t$, 
equality holds if and only if all the $x_i$ are the same. 
Hence \eqref{eq:mean2} rewritten as
$$\frac{\sum a_i^2}{t}=1=\sqrt{\frac{\sum a_i^4}{t}},$$ gives $a_1^2=\dots=a_t^2=1$. 
Therefore, $$\sum_{i=1}^t a_i=\sum_{i=1}^t a_i^3\in \Z.$$ Since $(1-\sqrt 2)t\notin \Z$, 
this contradicts
\eqref{eq:mean2} and there is no representation \eqref{eq:meanRep} of $y$.

\subsection*{Example \ref{expsd} revisited}
The truncated tracial moment matrix 
$$M_2(y)=\left(\begin{smallmatrix}
	1&0&0&1&1&1&1\\ 0&1&1&0&0&0&0\\ 0&1&1&0&0&0&0\\ 1&0&0&4&0&0&2\\
	1&0&0&0&2&1&0\\ 1&0&0&0&1&2&0\\ 1&0&0&2&0&0&4
\end{smallmatrix}\right)$$
is positive semidefinite but 
with respect to the index set $$(1,X,Y,X^2,XY,YX,Y^2),$$ $y$ has no tracial moment representation \eqref{rep}.
 
Assume $y_ w=\sum_{i=1}^N \lambda_i \Tr( w(A_1^{(i)},A_2^{(i)})) $ for some symmetric matrices $A_j^{(i)}$ and
$\lambda_i\in\R_{\geq0}$ with $\sum_i \lambda_i=1$.
Setting $$T^{(i)}:=\big(\Tr(u^*v( A_1^{(i)},A_2^{(i)}))\big)_{u,v}$$ we have 
$M_2(y)=\sum_{i=1}^N \lambda_i T^{(i)}$. Each $T^{(i)}$ is positive semidefinite, 
thus in particular $ T^{(i)}_{22}= T^{(i)}_{33}= T^{(i)}_{23}=:t_i$ holds for all $i=1,\dots,N$.
Let $d_i$ be the size of the symmetric matrices $A_j^{(i)}$, $j=1,2$. From
\begin{align*}
\frac1{d_i^2}\langle A_1^{(i)},A_1^{(i)}\rangle\langle A_2^{(i)},A_2^{(i)}\rangle&=
\Tr({A_1^{(i)}}^2)\Tr({A_2^{(i)}}^2)=t_i^2\\
&=(\Tr(A_1^{(i)}A_2^{(i)}))^2=\frac1{d_i^2}\langle A_1^{(i)},A_2^{(i)}\rangle^2
\end{align*}
we obtain by the Cauchy-Schwarz inequality that $A_1^{(i)}=\alpha_i A_2^{(i)}$ for some $\alpha_i\in \R$, 
$i=1,\dots,N$. But then we derive the contradiction
\begin{align*}
2&=M_2(y)_{55}=\sum_{i=1}^N\lambda_i T^{(i)}_{55}=\sum \lambda_i \Tr({A_1^{(i)}}^2{A_2^{(i)}}^2)=
	\sum \lambda_i\alpha_i^2\Tr({A_2^{(i)}}^4)\\&
	=\sum \lambda_i \Tr(A_1^{(i)}A_2^{(i)}A_1^{(i)}A_2^{(i)})=M_2(y)_{45}=1.
\end{align*}

\subsection*{Acknowledgments}
\small
Both authors thank Markus Schweighofer for insightful comments and suggestions. 
The second author also thanks Scott McCullough and Jiawang Nie for
enlightening discussions.

\end{document}